\documentclass[]{article}
\usepackage{amsmath,amsthm,amsfonts}
\usepackage[all]{xy}
\SelectTips{cm}{}

\newtheorem{theorem}{Theorem}[section]
\newtheorem{lemma}[theorem]{Lemma}

\newcommand{\relproj}[1]{%
\cat{relproj}(#1)}%
\newcommand{\directsum}{\oplus}%
\newcommand{\abmon}[1]{%
A^{+}(#1)}%
\newcommand{\fac}[1]{%
\cat{fac(#1)}}%
\newcommand{\homperp}[1]{%
\cat{Homperp}(#1)}%
\newcommand{\xymap}[3]{\xymatrix@1{
#1 \ar[r]^-{_{#3}}  & #2  }
}

\newcommand{\cat}[1]{\text{\textup{\textsf{#1}}}}

\newcommand{\fpmod}[1]{\cat{fpmod}(#1)}

\newcommand{\fgproj}[1]{\cat{proj}(#1)}
  
\newcommand{\tses}[3]{0\rightarrow #1\rightarrow #2 \rightarrow%
#3\rightarrow 0}%
\newcommand{\shortexactsequence}[3]{\xymatrix@1{
0 \ar[r]  & #1 \ar[r]  & #2 \ar[r]  & #3 \ar[r]  & 0  }
}
\newcommand{\map}[3]{#1 \colon #2 \to #3}
\newcommand{\tensor}{\otimes}

\newcommand{\Mod}[1]{\cat{Mod}(#1)}

\newcommand{\pullses}[5]{%
\xymatrix{%
0 \ar[r]  & #1 \ar[r] \ar@<0.5ex>@{-}[d] \ar@<-0.5ex> @{-}[d]  & #2%
\ar[r] \ar[d] & #3 \ar[r] \ar[d] & 0\\%
0 \ar[r]  & #1 \ar[r]  & #4 \ar[r]  & #5 \ar[r]  & 0  }%
}%

\newcommand{\pullsesalong}[6]{%
\xymatrix{%
0 \ar[r]  & #1 \ar[r] \ar@<0.5ex>@{-}[d] \ar@<-0.5ex> @{-}[d]  & #2%
\ar[r] \ar[d] & #3 \ar[r] \ar[d]^{#6} & 0\\%
0 \ar[r]  & #1 \ar[r]  & #4 \ar[r]  & #5 \ar[r]  & 0  }%
}%

\newcommand{\pushses}[5]{%
\xymatrix{%
0 \ar[r]  & #1 \ar[r] \ar[d]  & #2 \ar[r] \ar[d] & #3 \ar[r] \ar@<0.5ex>%
@{-}[d]  \ar@<-0.5ex> @{-}[d]   & 0\\%
0 \ar[r]  & #4 \ar[r]  & #5 \ar[r]  & #3 \ar[r]  & 0  }%
}%

\newcommand{\pushsesalong}[6]{%
\xymatrix{%
0 \ar[r]  & #1 \ar[r] \ar[d]^{#6}  & #2 \ar[r] \ar[d] & #3 \ar[r] \ar@<0.5ex>%
@{-}[d]  \ar@<-0.5ex> @{-}[d]   & 0\\%
0 \ar[r]  & #4 \ar[r]  & #5 \ar[r]  & #3 \ar[r]  & 0  }%
}%

\newcommand{\Qnloops}{%
\underset{_{n \text{ arrows}}}{%
\xymatrix{%
\cdot \ar@(dl,ul)[] \ar@(ur,dr)[] \ar@(ul,ur)@{.>}[]%
}}}%
\newcommand{\QnKronecker}{%
\underset{_{n \text{ arrows}}}{\xymatrix{%
\cdot \ar@<1.2ex>[r] \ar@<-1.2ex>[r] \ar@<0.9ex>@{}[r] |{\cdot}%
\ar@<0.3ex>@{}[r] |{\cdot} \ar@<-0.3ex>@{}[r] |{\cdot}%
\ar@<-0.9ex>@{}[r] |{\cdot} &\cdot%
}}%
}%
\newcommand{\QnA}{%
\underset{_{n \text{ vertices}}}{\xymatrix{%
\cdot \ar[r] &\cdot\  \ar@{}[r] |{\cdots\cdots\cdots} &\ \cdot \ar[r] &\cdot%
}}%
}%

\DeclareMathOperator{\Hom}{Hom}

\DeclareMathOperator{\im}{im}

\DeclareMathOperator{\Ext}{Ext}
\DeclareMathOperator{\Tor}{Tor}

\begin{document}

\title{Universal localisations of hereditary rings}
\author{Aidan Schof\i eld}
\maketitle

\begin{abstract}
We describe all possible universal localisations of a hereditary ring
in terms of suitable full subcategories of the category of finitely
presented modules. For these universal localisations we then identify
the category of finitely presented bound modules over the universal
localisation as being equivalent to a certain full subcategory of
the category of finitely presented bound modules over the original
ring. We also describe the abelian monoid of finitely generated
projective modules over the universal localisation. 
\end{abstract}  

\section{Introduction} 
\label{sec:intro} 

The purpose of this paper is to study the universal localisations of a
hereditary ring. We shall describe all possible universal
localisations and for each of these we determine the kernel of the
homomorphism from $R$ to its universal localisation $R_{\Sigma}$ and
the kernel of the homomorphism from an $R$ module $M$ to $M \tensor
R_{\Sigma}$, the structure of a finitely presented module over
$R_{\Sigma}$ as an $R$ module and the structure of the category of
finitely presented modules over the universal localisation. 

Over a hereditary ring, every finitely presented module is a direct
sum of a finitely generated projective module and a bound module (a
module $M$ such that $\Hom(M,R) =0$). The universal localisation of a
hereditary ring is a hereditary ring and we show that the category of
finitely presented bound modules over $R_{\Sigma}$ is equivalent to a
full subcategory of $\fpmod{R} $ defined in terms of the set of maps
$\Sigma$. It remains to understand the finitely generated projective
modules over the universal localisation and we are able to describe
the abelian monoid of isomorphism classes of finitely generated
projective modules over the universal localisation in terms of a
suitable subcategory of the category of finitely presented modules
over $R$, $\fpmod{R}$.

This paper follows on from \cite{Scho07} which should be read before
this one.

This paper should be regarded as extending P. M. Cohn's and others'
investigation of the universal localisation of firs (see
\cite{CohnFIR}) . We find that essentially everything carries over to
this more general situation and our results clarify the problem of
determining when any universal localisation of a fir is itself a fir
by looking at modules instead of matrices. In fact, the methods
developed here will often allow us to determine the isomorphism
classes of finitely generated projective modules over such a universal
localisation of a fir which is a task unattempted by previous theory.

In a subsequent paper, we shall see
that these more general hereditary rings and their universal
localisations are of great interest since we shall show that the right
perpendicular category to a vector bundle over a smooth projective
curve is equivalent to $\Mod{R}$ for a suitable (usually
noncommutative) hereditary ring. The description of these rings and
their investigation heavily use the techniques of this paper.

In the next section, we characterise all possible universal
localisations of a hereditary ring. The following section describes
the finitely presented bound modules over the universal localisation
in terms of a suitable subcategory of finitely presented modules and
the last section studies the finitely generated projective modules
over the universal localisation again in terms of a suitable
subcategory of finitely presented modules.

Throughout this paper, a hereditary ring will mean a left and right
hereditary ring.

\section{Characterising the universal localisations}   
\label{sec:charul} 

In the paper \cite{Scho07}, we saw that a universal localisation
$R_{\Sigma}$ of a ring $R$ was best studied in terms of the full
subcategory of finitely presented modules of homological dimension at
most $1$ that are left perpendicular to the $R_{\Sigma}$ modules. We
say that a module $M$ of homological dimension at most $1$ is left
perpendicular to $N$ if and only if $\Hom(M,N) =0=\Ext(M,N) $.  We
call this the category of $R_{\Sigma}$ trivial modules,
$\cat{S}(\Sigma)$.  The best cases in that paper were shown to be
those where the kernel of a map in $\cat{S}(\Sigma)$ was itself a
factor of another such module. In the case of a hereditary ring, we
shall show that a better result holds. This category must be an exact
abelian subcategory of $\fpmod{R}$. By this we mean that it is an
abelian category and the inclusion is exact. Equivalently, it is an
additive subcategory closed under images, kernels and cokernels. Of
course, it is also closed under extensions and we call a full exact
abelian subcategory of $\fpmod{R}$ closed under extensions a
well-placed subcategory of $\fpmod{R}$. Our main result is that the
set of well-placed subcategories of $\fpmod{R}$ parametrise the
universal localisations of a hereditary ring $R$. These categories
have further structure which we end this section by describing since
it is important for the subsequent discussion of the finitely
presented bound modules and the finitely generated projective modules
over the universal localisation. We also give a precise description of
the structure of a finitely presented module over the universal
localisation as a directed union of modules over $R$.

In general, any universal localisation of a ring need not be a
universal localisation at an injective set of maps between finitely
generated projective modules. One can reduce to this case in principle
by factoring out a suitable ideal which may however be difficult to
determine in practice. For a hereditary ring we can avoid this problem
as our first lemma shows.

\begin{lemma}
\label{lem:reducetoinj}
Let $R$ be a right hereditary ring and let $\Sigma$ be a set of maps
between finitely generated projective modules over $R$. Then there
exists a set of injective maps between finitely generated projective
modules $\Sigma'$ such that $R_{\Sigma}\cong R_{\Sigma'}$.
\end{lemma}
\begin{proof}
  Let $\map{f}{P}{Q}$ be a map in $\Sigma$. Let $I=\im(f) $ and
  $K=\ker(f) $ . Then $f=pf'$ where the induced map $\map{p}{P}{I}$ is
  split surjective with kernel $K$. Therefore, $p\tensor R_{\Sigma}$
  and $f' \tensor R_{\Sigma}$ are invertible. Conversely, if we invert
  $p$ and $f'$ then we must also invert $f$. However, adjoining a
  universal inverse to the split surjective map $p$ is equivalent to
  adjoining an inverse to its left inverse, a split injective map
  $\iota$ from $I$ to $P$ with cokernel $K$ and vice versa. Thus we
  may replace $\Sigma$ by $\Sigma_{1}=\Sigma-\{f\} \cup \{f',
  \iota\}$. Applying this to each non-injective map in $\Sigma$ gives
  us a set of injective maps $\Sigma'$ such that $R_{\Sigma}\cong
  R_{\Sigma'}$.
\end{proof}

We have already seen in \cite{Scho07} that whenever $\Sigma$ is a set
of injective maps between finitely generated projective modules, it is
useful to replace consideration of $\Sigma$ by the full subcategory of
$R_{\Sigma}$-trivial modules. This is the full subcategory of
$\fpmod{R}$ whose objects are those modules of homological dimension
at most $1$ such that their presentations are inverted by $\tensor
R_{\Sigma}$ or equivalently those $R$ modules $M$ of homological
dimension at most $1$ such that $\Hom(M,\ ) $ and $\Ext(M,\ ) $ vanish
on $R_{\Sigma}$ modules (see theorem 5.2 in \cite{Scho07}). In that
paper, we showed that particularly sensible results held whenever this
category has the property that kernels of maps in this category are
torsion with respect to the torsion theory generated by this
category. In fact, we can do better here. Over a right hereditary
ring, the category of $R_{\Sigma}$ trivial modules is an abelian
subcategory of $\fpmod{R}$ whose inclusion in $\fpmod{R}$ is exact.

\begin{theorem}
\label{th:Strivisab}
Let $R$ be a right hereditary ring and let $\Sigma$ be a set of maps
between finitely generated projective modules over $R$. Then the
category of $R_{\Sigma}$ trivial modules, $\cat{S}(\Sigma)$, is an
exact abelian subcategory of $\fpmod{R}$ closed under extensions.
\end{theorem}
\begin{proof}
We need to show that $\cat{S}(\Sigma)$ is closed under
images. However, for any factor $F$ of a module in $\cat{S}(\Sigma)$,
$\Hom(F,\ ) $ vanishes on $R_{\Sigma}$ modules and for any submodule
$S$ of a module in $\cat{S}(\Sigma)$, $\Ext(S,\ ) $ vanishes on
$R_{\Sigma}$ modules (because $R$ is right hereditary). Since the
image of a map in $\cat{S}(\Sigma)$ is both a submodule and a factor
module, we see that $\cat{S}(\Sigma)$ is closed under images. The rest
follows once we note that given a short exact sequence
$\tses{A}{B}{C}$ where two of $A,B,C$ lie in $\cat{S}(\Sigma)$ then so
does the remaining term. 
\end{proof}

Thus to any universal localisation of a right hereditary ring we may
associate a full exact abelian subcategory of $\fpmod{R}$ closed under
extensions. We recall that we have named such a subcategory a
well-placed subcategory.

\begin{theorem}
\label{th:charulbywpc}
Let $R$ be a right hereditary ring. Then the universal localisations
of $R$ are parametrised by the well-placed subcategories of
$\fpmod{R}$ by associating to a universal localisation $R_{\Sigma}$
the category of $R_{\Sigma}$ trivial modules.
\end{theorem}
\begin{proof}
We have already seen one half of this bijection in the previous
theorem and the following remark. Conversely, let $\cat{E}$ be a
well-placed subcategory of $\fpmod{R}$. Then in the terminology of
\cite{Scho07} $\cat{E}$ is certainly a pre-localising category such
that the kernel of any map in $\cat{E}$ is torsion with respect to the
torsion theory generated by $\cat{E}$. Therefore, we may form the
universal localisation $R_{\cat{E}}$ and by theorem 5.6 of
\cite{Scho07}, the category of $R_{\cat{E}}$-trivial modules is just
$\cat{E}$. Thus we have the stated bijection. 
\end{proof}

Since all universal localisations of a right hereditary ring are
parametrised by well-placed subcategories of $\fpmod{R}$ we shall use
the notation $R_{\cat{E}}$ for the universal localisation
corresponding to the well-placed subcategory $\cat{E}$. Note that the
category of $R_{\cat{E}}$ modules considered as $R$ modules can be
identified with the category $\cat{E}^{\perp}$ which is the full
subcategory of $R$ modules $M$ such that $\Hom(E,M) =0= \Ext(E,M) $
for every $E\in \cat{E}$. 

A well-placed subcategory $\cat{E}$ of $\fpmod{R}$ for a right
hereditary ring is fairly easy to understand. A finitely presented
module over a right hereditary ring is a direct sum of a bound module
and a finitely generated projective module. Clearly, a well-placed
subcategory of $\fpmod{R}$ is closed under direct summands and so
these summands lie in $\cat{E}$. Any finitely presented bound module
over a right hereditary ring satisfies the ascending chain condition
on bound submodules by corollary 5.1.7 of \cite{CohnFIR}. Thus in the
case where $\cat{E}$ contains no projective module (which is clearly
equivalent to assuming that $R$ embeds in the universal localisation),
we know that $\cat{E}$ is a Noetherian abelian category. The best
results occur when we assume that $R$ is a hereditary ring by which we
mean that is both left and right hereditary.

For the rest of the paper we shall usually assume that $R$ embeds in
the universal localisation. This is equivalent to assuming that
$\cat{E}$ contains no nonzero finitely generated projective module. We
can reduce to this stuation by factoring out by the trace ideal of the
finitely generated projective modules in $\cat{E}$. This is not ideal
but the results are most simply demonstrated in this context.

\begin{theorem}
\label{th:wpisfl}
Let $R$ be a hereditary ring and let $\cat{E}$ be a well-placed
subcategory of $\fpmod{R}$  of modules of homological dimension
$1$. Then $\cat{E}$ is a finite length category.
\end{theorem}
\begin{proof}
We assumed that every module in $\cat{E}$ has homological dimension
$1$; therefore no such module can have a nonzero projective summand
since $\cat{E}$ is closed under direct summands. It follows that all
modules in $\cat{E}$ are bound. However, by theorem 5.2.3 of
\cite{CohnFIR}, any bound module over a hereditary ring satisfies the
descending chain condition on bound submodules as well as the
ascending chain condition and so $\cat{E}$ is a finite length
category. 
\end{proof}

We say that a set $S$ of modules is $\Hom$-perpendicular if and only
if for $T,U\in S$, and a map $\map{f}{T}{U}$ then $f$ is invertible or
$0$. Given a $\Hom$-perpendicular set of bound finitely presented
modules, its extension closure is a well-placed subcategory of
$\fpmod{R}$. Conversely, given a well-placed subcategory of
$\fpmod{R}$ of modules of homological dimension $1$, it is a finite
length category and its simple objects form a $\Hom$-perpendicular set
of bound finitely presented modules. Moreover, since $\cat{E}$ is
closed under extensions and every object has a composition series in
$\cat{E}$, $\cat{E}$ is the closure of its set of simple objects under
extensions. We summarise this discussion in the following theorem.

\begin{theorem}
\label{th:charwp}
Let $R$ be a hereditary ring. Then its universal localisations into
which it embeds are parametrised by the set of $\Hom$-perpendicular
sets of isomorphism classes of finitely presented bound modules.  This
bijection arises by assigning to the universal localisation
$R_{\Sigma}$ the set of simple objects in the category of $R_{\Sigma}$
trivial modules. 
\end{theorem}

We use this to obtain a better description of a module $M\tensor
R_{\cat{E}}$ for a finitely presented module $M$ over the hereditary
ring $R$. For any universal localisation, we showed (see theorem 4.2
of \cite{Scho07}) how to find the $R$ module structure of $M \tensor
R_{\Sigma}$ as a direct limit of finitely presented modules over
$R$. The maps in this direct limit are not all injective but in the
case where $R$ is a hereditary ring we can represent $M\tensor
R_{\cat{E}}$ as a direct limit of finitely presented modules where the
maps in the system are injective.

\begin{theorem}
\label{th:structfpind}
Let $R$ be a hereditary ring and let $\cat{E}$ be a well-placed
subcategory of $\fpmod{R}$ of modules of homological dimension
$1$. Let $M$ be a finitely presented module over $R$. Then the kernel
of the map from $M$ to $M\tensor R_{\cat{E}}$ is finitely generated. 

If $M$ embeds in $M\tensor R_{\cat{E}}$ then as $R$ module $M\tensor
R_{\cat{E}}$ is a directed union of modules $M_{t}\supset M$ where
$M_{t}/M \in \cat{E}$. 
\end{theorem}
\begin{proof}
By theorem 5.5 of \cite{Scho07} we know that the kernel of the map
from $M$ to $M\tensor R_{\cat{E}}$ is the torsion submodule of $M$
with respect to the torsion theory generated by $\cat{E}$. However,
$M$ satisfies the ascending chain condition on bound submodules and so
the torsion submodule must be finitely generated. 

Now suppose that $M$ embeds in $M\tensor R_{\cat{E}}$ so that $M$ is
torsion-free with respect to the torsion theory generated by
$\cat{E}$. By theorem 4.2 of \cite{Scho07}, $M\tensor R_{\cat{E}}$ is
a direct limit of modules $M_{s}$ where the map from $M$ to $M_{s}$ is
in the severe right Ore set generated by a set of presentations of the
modules in $\cat{E}$. Each of these factors as a good pushout
followed by a good surjection (see \cite{Scho07} for these terms). In
our present case, a good pushout is injective.

Thus we have a short exact sequence $\tses{M}{N_{s}}{E_{s}}   $ where
$E_{s} \in \cat{E}$ and a surjection from $N_{s}$ to $M_{s}$ whose
kernel is torsion with respect to the torsion theory generated by
$\cat{E}$. The inclusion of $M$ in $N_{s}$ is inverted by $\tensor
R_{\cat{E}}$ and so the image of $N_{s}$ in $M\tensor R_{\cat{E}}$ is
$N_{s}/T_{s}$ where $T_{s}$ is the torsion submodule of $N_{s}$ with
respect to the torsion theory generated by $\cat{E}$ and so $T_{s}$ is
finitely generated by the first paragraph of this proof. We note that
$M\tensor R_{\cat{E}}$ is the directed union of the modules
$N_{s}/T_{s}$ and we wish to show that the cokernel of the inclusion
of $M$ in each of these modules lies in $\cat{E}$.  

We choose some $G_{s} \in \cat{E}$ that maps onto $T_{s}$ and consider
the induced map from $G_{s}$ to $E_{s}$. Then the image $I_{s}$ must
be in $\cat{E}$ and so must the kernel $K_{s}$. Under the map from
$G_{s}$ to $N_{s}$ the image of $K_{s}$ must lie in $M$ and so must be
zero since $\Hom(E,M) =0$ for every $E\in \cat{E}$. It follows that
$T_{s}\cong I_{s}$ and so we have a short exact sequence
$\tses{M}{N_{s}/T_{s}}{E_{s}/I_{s}}$  where $E_{s}/I_{s} \in \cat{E}$
as we set out to prove.  
\end{proof}

This allows us to give a simple criterion for when two finitely
presented module over $R$ induce up to isomorphic modules over
$R_{\cat{E}}$ which is better than the results we prove for general
universal localisations in \cite{Scho07}. 

\begin{theorem}
\label{th:chariso}
Let $R$ be a hereditary ring and let $\cat{E}$ be a well-placed
subcategory of $\fpmod{R}$ all of whose modules are bound. Let $M$ and
$N$ be finitely presented module over $R$ that are torsion-free with
respect to the torsion theory generated by $\cat{E}$. Then $M\tensor
R_{\cat{E}}$ is isomorphic to $N \tensor R_{\cat{E}}$ if and only if
there exist a module $L$ and short exact sequences $\tses{M}{L}{E}$
and $\tses{N}{L}{F}$ where $E,F \in \cat{E}$. 
\end{theorem}
 \begin{proof}
Certainly, the existence of such short exact sequences implies that
$M\tensor R_{\cat{E}} \cong N \tensor R_{\cat{E}}$.

Now suppose that $M\tensor R_{\cat{E}} \cong N \tensor R_{\cat{E}}$.
Since $M\tensor R_{\cat{E}}$ is a directed union of modules $M_{t}$
where $M\subset M_{t}$ and $M_{t}/M \in \cat{E}$ (by the last
theorem), the isomorphism of $M\tensor R_{\cat{E}}$ and $N \tensor
R_{\cat{E}}$ gives an embedding of $N$ in $M\tensor R_{\cat{E}}$ and
hence of $N$ in some $M_{t}$; moreover, the inclusion of $N$ in
$M_{t}$ becomes an isomorphism under $\tensor R_{\cat{E}}$. 

Since $\Tor_{1}^{R}(M\tensor R_{\cat{E}}, R_{\cat{E}}) =0$ and
$M_{t}\subset M\tensor R_{\cat{E}}$, it follows that
$\Tor_{1}^{R}(M_{t}, M\tensor R_{\cat{E}}) =0$ . So applying $\tensor
R_{\cat{E}}$ to the short exact sequence $\tses{N}{M_{t}}{G}$ shows
that $G \tensor R_{\cat{E}} =0= \Tor_{1}^{R}(G, R_{\cat{E}}) $ from
which we conclude that $G$ must be an $R_{\cat{E}}$-trivial module and
hence $G \in \cat{E}$. Thus our proof is complete.  
\end{proof}

\section{Bound modules over a universal localisation} 
\label{sec:bound} 

In this section, $R$ is a hereditary ring and $\cat{E}$ is a
well-placed subcategory of $\fpmod{R}$. We shall be interested in
describing the category of finitely presented bound modules over
$R_{\cat{E}}$. We shall assume that $\cat{E}$ contains no projective
modules since we can replace $R$ by $R/T$ where $T$ is the trace ideal
of the projective modules in $\cat{E}$ which is still a hereditary
ring and we replace $\cat{E}$ by its image under $\tensor R/T$.

In order to describe our theorem we introduce some terminology. Let
$S$ be a set of modules or $\cat{S}$ the full subcategory of modules
isomorphic to a module in $S$. Then the $\Hom$-perpendicular category
to $S$ (or to $\cat{S}$), written as $\homperp{S}$ (or
$\homperp{\cat{S}}$) is the full subcategory of bound modules $X$ such
that $\Hom(X,M) =0= \Hom(M,X) $ for every $M\in S$ (or $M\in
\cat{S}$). We shall prove that the category of finitely presented
bound modules over $R_{\cat{E}}$ is equivalent to $\homperp{\cat{E}}
$.

Befor we begin, we should not that certain finitely presented modules
over $R$ cannot induce to bound modules over the universal
localisation. 

\begin{lemma}
\label{lem:subnotbound}
Let $M$ be a strict submodule of $E$, a simple object in
$\cat{E}$. Then $M\tensor R_{\cat{E}}$ is a nonzero projective module
over $R_{\cat{E}}$ 
\end{lemma}
\begin{proof}
  We note that $\Ext(M,\ ) =0$ on $R_{\cat{E}}$ modules because $R$ is
  hereditary and $M\subset E$ for which by assumption $\Ext(E,\ ) $
  vanishes on $R_{\cat{E}}$ modules. However, $\Hom(M,\ ) $ cannot
  vanish on $R_{\cat{E}}$ modules since $M$ is not in $\cat{E}$ and
  theorem 5.6 shows that $\cat{E}$ is the category of
  $R_{\cat{E}}$-trivial modules. But $\Hom_{R_{\cat{E}}}(M\tensor
  R_{\cat{E}},\ ) = \Hom(M,\ ) $ on $R_{\cat{E}}$ modules so that
  $M\tensor R_{\cat{E}}$ cannot be $0$. 
\end{proof}

We begin by showing that that every finitely presented bound module
over the universal localisation is induced from a module in
$\homperp{\cat{E}}$.  

\begin{theorem}
\label{th:thecore}
Let $R$ be a hereditary ring and $\cat{E}$ a well-placed
subcategory of $\fpmod{R}$ all of whose modules are bound. Then every
finitely presented bound module over $R_{\cat{E}}$ is isomorphic to a
module of the form $M\tensor R_{\cat{E}}$ where $M\in
\homperp{\cat{E}}$. 
\end{theorem}
\begin{proof}
Every finitely presented module over $R_{\cat{E}}$ is isomorphic to a
module of the form $M\tensor R_{\cat{E}}$ where $M\in \fpmod{R}$ and
$\Hom(E,M) =0$ for every $E\in \cat{E}$ by theorem \ref{th:structfpind}.
Clearly, $M$ is a bound module since $P \tensor R_{\cat{E}}$ is
nonzero for any projective module and we are assuming that $M\tensor
R_{\cat{E}}$ is bound. 

Now suppose that $\Hom(M,E) \neq 0$ for some module $E\in
\cat{E}$. Then since $\cat{E}$ is a finite length category we may
assume that $E$ is a simple object in $\cat{E}$. If the homomorphism
is not surjective then we obtain a short exact sequence
$\tses{K}{M}{F}$ where $F$ is a strict submodule of $E$ and so $F
\tensor R_{\cat{E}}$ is nonzero projective by lemma
\ref{lem:subnotbound}. Therefore, applying $\tensor R_{\cat{E}}$ to
our short exact sequence shows that $M\tensor R_{\cat{E}}$ is not
bound. This contradiction implies that we have a surjection from $M$
to $E$ whenever $E$ is a simple object in $\cat{E}$ and there is a
nonzero homomorphism from $M$ to $E$. Thus we have a short exact
sequence $\tses{M_{1}}{M}{E_{1}}$ whenever $M$ has a nonzero
homomorphism to some module in $\cat{E}$. From this short exact
sequence, we see that $M \tensor R_{\cat{E}} \cong M_{1} \tensor
R_{\cat{E}}$. Applying this argument as many times as necessary we
either obtain a submodule $M_{n}\subset M$ such that $M_{n} \tensor
R_{\cat{E}} \cong M\tensor R_{\cat{E}}$ and $\Hom(M_{n},E) =0$ for
every $E\in \cat{E}$ or else a strictly descending sequence of
submodules of $M$, $M\supset M_{1}\supset \dots M_{n}\dots $ where
$M/M_{n}\in \cat{E}$ and $M_{n} \tensor R_{\cat{E}} \cong M\tensor
R_{\cat{E}}$. However, any bound module over a hereditary ring
satisfies the descending chain condition on submodules whose factor is
bound and so the second case cannot arise. Thus we have some submodule
$M'\subset M$ such that $M' \tensor R_{\cat{E}} \cong M\tensor
R_{\cat{E}}$ and $\Hom(M',E) =0$ for every $E\in \cat{E}$. But $M'$
must also satisfy $\Hom(E,M) =0$ for every $E\in \cat{E}$. So $M'\in
\homperp{\cat{E}}$ which completes our proof.
\end{proof}

Next we restrict the $R_{\cat{E}}$ homomorphisms from a module
$M\tensor R_{\cat{E}}$ where $M \in \homperp{\cat{E}}$. 

\begin{lemma}
\label{lem:restrict}
Let $R$ be a hereditary ring and $\cat{E}$ a well-placed subcategory
of $\fpmod{R}$ all of whose modules are bound. Let $M\in
\homperp{\cat{E}}$ and let $N$ be some finitely presented module over
$R$ torsion-free with respect to the torsion theory generated by
$\cat{E}$. Let $\map{\phi}{M\tensor R_{\cat{E}}}{N\tensor
  R_{\cat{E}}}$ be a homomorphism of $R_{\cat{E}}$ modules. Then $\phi
= f \tensor R_{\cat{E}}$ for some $\map{f}{M}{N}$.
\end{lemma}
\begin{proof}
We showed in theorem \ref{th:structfpind} that 
$N \tensor R_{\cat{E}}$ is isomorphic to a direct
limit of modules $N_{t}\supset N$ such that $N_{t}/N \in \cat{E}$. So
the image of $\phi$ must lie in some $N_{t}$. Since $N_{t}/N \in
\cat{E}$ and $M\in \homperp{\cat{E}}$, $\Hom(M, N_{t}/N) =0$ and so
the image of $M$ lies in $N$. Let $\map{f}{M}{N}$ be this induced
map. Then $\phi- f \tensor R_{\cat{E}}$ vanishes on $M$ and so must be
the zero map as required.
\end{proof}

We note a consequence of this.

\begin{lemma}
\label{lem:indfromhompisbound}
Let $M \in \homperp{\cat{E}}$. Then $M\tensor R_{\cat{E}}$ is bound
\end{lemma}
\begin{proof}
From the last lemma we see that $\Hom_{R_{\cat{E}}}(M\tensor
R_{\cat{E}}, R_{\cat{E}}) = \Hom_{R}(M,R) =0$. 
\end{proof}

We conclude the main theorem of this section. 

\begin{theorem}
\label{th:boundeqhomperp}
Let $R$ be a hereditary ring and $\cat{E}$ a well-placed subcategory
of $\fpmod{R}$ all of whose modules are bound. Then $\tensor
R_{\cat{E}}$ induces an equivalence of categories between
$\homperp{\cat{E}}$ and the full subcategory of finitely presented
bound modules over $R_{\cat{E}}$.
\end{theorem}
\begin{proof}
  We have already seen in theorem \ref{th:thecore} and lemma
  \ref{lem:indfromhompisbound} that $\tensor R_{\cat{E}}$ induces a
  functor from $\homperp{\cat{E}}$ to the full subcategory of finitely
  presented bound modules over $R_{\cat{E}}$ that is
  surjective on isomorphism classes of objects.

  However, by lemma \ref{lem:restrict}, any homomorphism from
  $M\tensor R_{\cat{E}}$ to $N\tensor R_{\cat{E}}$ for $M,N \in
  \homperp{\cat{E}}$ must be of the form $f \tensor R_{\cat{E}}$ for
  some $\map{f}{M}{N}$. Thus $\tensor R_{\cat{E}}$ is surjective on
  $\Hom$-sets as a functor from $\homperp{\cat{E}}$ to the full
  subcategory of finitely presented bound modules over
  $R_{\cat{E}}$. However, it is visibly injective since $N$ is a
  submodule of $N\tensor R_{\cat{E}}$. Thus $\tensor R_{\cat{E}}$
  induces the stated equivalence.
\end{proof}

In some ways, this result is very surprising although its proof is
simple enough. One consequence is that the category of finitely
presented bound modules over the free associative algebra over a field
is equivalent to a category of finite dimensional representations over
a generalised Kronecker quiver since a $2$ by $2$ matrix ring over the
free associative algebra is a universal localisation of the path
algebra of a generalised Kronecker quiver. Perhaps more shocking is
that the category of finitely presented bound modules over the
universal algebra with an isomorphism from the free module of rank $1$
to the free module of rank $2$ is equivalent to a category of finite
dimensional representations over a generalised Kronecker quiver
because it too is a universal localisation of a free algebra and hence
Morita equivalent to a universal localisation of the path algebra of a
generalised Kronecker quiver. Thus the category of finitely presented
bound modules over  a hereditary ring has no knowledge of the
pathology of the finitely generated projective modules.

\section{Finitely generated projective modules} 
\label{sec:proj} 

We wish to calculate the finitely generated projective modules over
the universal localisation of a hereditary ring. We begin by pointing
out that the map from $K_{0}(R) $ to $K_{0}(R_{\cat{E}})$ is
surjective which is not guaranteed for general rings.

\begin{lemma}
\label{K0ok}
The map from $K_{0}(R)$ to $K_{0}(R_{\cat{E}})$ is surjective.
\end{lemma}
\begin{proof}
Every finitely presented module over $R_{\cat{E}}$ is induced. 
So given a finitely generated projective module over $R_{\cat{E}}$ we
can write it in the form $M\tensor R_{\cat{E}}$ for some finitely
presented module $M$ that
embeds in $M\tensor R_{\cat{E}}$. It follows that $\Tor_{1}^{R}(M,
R_{\cat{E}}) =0$. We consider a short exact sequence $\tses{P}{Q}{M}$
where $P,Q$ are finitely generated projective modules over
$R$. Tensoring gives a short exact sequence $\tses{P \tensor
  R_{\cat{E}}}{Q \tensor R_{\cat{E}}}{M\tensor R_{\cat{E}}}$ and
completes the proof of the lemma.
\end{proof}
 
Thus we know what the $K$-theory of the universal localisation is and
so we know the finitely generated projective modules stably but we can
get far more precise information; in principle, we can determine the
abelian monoid of the finitely generated projective modules up to
isomorphism. In order to describe our main result, we need to
introduce a couple of categories and a construction.

Let $R$ be a hereditary ring and let $\cat{E}$ be a well-placed
subcategory of $\fpmod{R}$. Our first category is the full subcategory
of $\fpmod{R} $ whose objects are those modules $M$ such that
$\Ext(M,\ ) $ vanishes on $E^{\perp}$, the category of $R_{\cat{E}}$
modules in $\Mod{R} $; we call this $\relproj{E^{\perp}}$, the category of
finitely presented modules projective with respect to
$E^{\perp}$. This category is clearly closed under extensions and
submodules and contains the finitely generated projective modules and
$\cat{E}$; we shall see shortly that it is the smallest such full
subcategory.  Our second category is the full subcategory of
$\fpmod{R} $ whose objects are those modules $M$ such that $\Hom(M,\ )
$ vanishes on $E^{\perp}$; equivalently, $M\tensor R_{\cat{E}}=0$ from
which we see that it is precisely the full subcategory of $\fpmod{R} $
whose objects are the finitely presented modules that are factors of
modules in $\cat{E}$; equivalently, these are the finitely presented
torsion modules with respect to the torsion theory generated by
$\cat{E}$. We shall refer to this category $\fac{E} $ as the category
of factors of $\cat{E}$. 
Both of these categories are exact categories where the
short exact sequences are all short exact sequences in $\fpmod{R} $
whose objects lie in the relevant full subcategory.

Given an exact category $\cat{F} $ we associate to this an abelian
monoid $\abmon{\cat{F}} $ generated by the isomorphism classes of
objects in $\cat{F}$ and whose relations are given by the short exact
sequences in $\cat{F}$; that is, given a short exact sequence
$\tses{A}{B}{C}$  in $\cat{F}$, we have a relation $[B]=[A]+[C]$. 
In the special case where $\cat{F}$ is $\fgproj{R} $,
$\abmon{\fgproj{R}}$ is also known as $P_{\directsum}(R) $ and here
the elements of the abelian monoid are just the isomorphism classes of
objects in $\fgproj{R}$. In general, however, the map from the set of
isomorphism classes of objects in $\cat{F}$ to $\abmon{\cat{F}}$ is
only surjective.

Since $R$ is a hereditary ring we can show easily that
$\Tor_{1}^{R}(M, R_{\cat{E}}) =0$ for $M$ in $\relproj{E^{\perp}}$ so that
$\tensor R_{\cat{E}}$ is an exact functor from $\relproj{E^{\perp}}$ to
$\fpmod{R_{\cat{E}}}$ and in fact as we shall show to
$\fgproj{R_{\cat{E}}}$. Similarly, since $\tensor R_{\cat{E}}$
vanishes on $\fac{E}$, $\Tor_{1}^{R}(\ , R_{\cat{E}}) $ is an exact
functor from $\fac{E}$ to $\Mod{R} $ and in fact to
$\fgproj{R_{\cat{E}}}$. In the first case, the map on objects will be
shown to be surjective and hence we have a surjective map of abelian
monoids from $\abmon{\relproj{E^{\perp}}} $ and a map from $\abmon{\fac{E}}
$ to $P_{\directsum}(R_{\cat{E}}) $. In both cases, we shall show that
the additional relations in each case take the form $[E]=0$ for every
$E\in \cat{E}$.

Although this answer is useful, it helps in practice to have other
techniques available. First of all, one can show that every finitely
generated projective module over $ R_{\cat{E}}$ is isomorphic to a
direct sum of modules of the form $ S \tensor R_{\cat{E}}$ where
either $S$ is a submodule of some simple object in $\cat{E}$ (recall
that $\cat{E}$ is a finite length category) or else $S$ is itself a
finitely generated projective module. Secondly we show that that every
module in $\relproj{E^{\perp}}$ has a finite filtration such that the
subfactors are either finitely generated projective or else have the
property that any homomorphism to a module in $\cat{E}$ must be
injective. This last kind of module is sensible to work with since
these are the modules for which $S \tensor R_{\cat{E}}$ is finitely
generated projective and it is possible that $S \tensor R_{\cat{E}}$
is decomposable. If $S$ and $T$ are two such modules, we find a useful
way to investigate whether or not $S \tensor R_{\cat{E}}$ and $T
\tensor R_{\cat{E}}$ are isomorphic (which is a little technical to
describe at this moment). This is most useful when we know that there
are indecomposable finitely generated projective modules over the
universal localisation.

We begin with our first description of $P_{\directsum}(R_{\cat{E}})$.
We need to develop the properties of the categories $\relproj{E^{\perp}}$
and $\fac{E}$. 

We first need to show that the modules in $\relproj{E^{\perp}}$ induce up to
finitely generated projective modules over $R_{\cat{E}}$.

\begin{lemma}
\label{lem:PEptoproj}
Let $M\in \relproj{E^{\perp}}$. Then $M\tensor R_{\cat{E}}$ is a finitely
generated projective module over $ R_{\cat{E}}$ and this module is $0$
if and only if $M\in \cat{E}$. More generally, the kernel of the
homomorphism from $M$ to $M\tensor R_{\cat{E}}$ is in $\cat{E}$ and
the image of $M$ in $M\tensor R_{\cat{E}}$ also lies in
$\relproj{E^{\perp}}$.
\end{lemma}
\begin{proof}
We note that $\Hom_{R_{\cat{E}}}(M\tensor R_{\cat{E}},\ ) \cong
\Hom_{R}(M,\ ) $ on $E^{\perp}$ which is exact by assumption. So
$M\tensor R_{\cat{E}}$ is a projective module. It is finitely
generated because $M$ is. 

Further, $M\tensor R_{\cat{E}}=0$ if and only if $\Hom(M,\ ) $
vanishes on $E^{\perp}$. However, $\Ext(M,\ ) $ vanishes by
assumption. Theorem 5.6 of \cite{Scho07} now completes the proof that
this forces $M$ to lie in $\cat{E}$.

Since $R$ is coherent, a finitely generated submodule of $M$ is
finitely presented and because $R$ is hereditary, $M$ satisfies the
ascending chain condition on bound submodules. Therefore, the maximal
torsion submodule $T$ of $M$ (with respect to the torsion theory
generated by $\cat{E}$) is finitely presented. Since $\relproj{E^{\perp}}$
is closed under submodules, $T\in \relproj{E^{\perp}}$ and since $T$ is
torsion with respect to the torsion theory generated by $\cat{E}$,
$\Hom(T,\ ) =0$ on $E^{\perp}$ and the previous paragraph completes
the proof that $T$ lies in $\cat{E}$. Finally, we know by theorem 4.7 of
\cite{Scho1} that $\Ext_{R}(\ ,\ ) = \Ext_{R_{\cat{E}}}(\ ,\ )$ for
$R_{\cat{E}}$ modules and so, $\Ext_{R}(M\tensor R_{\cat{E}},\ ) =
\Ext_{R_{\cat{E}}}(M\tensor R_{\cat{E}},\ ) = 0$ and so any finitely
presented submodule and in particular the image of $M$ in $M\tensor
R_{\cat{E}}$ must be in $\relproj{E^{\perp}}$.
\end{proof}

Now we can characterise the category $\relproj{E^{\perp}}$. 

\begin{lemma}
\label{lem:charPEp}
The category $\relproj{E^{\perp}}$ is the smallest full subcategory of
$\fpmod{R}$ containing $\fgproj{R}$ and $\cat{E}$ that is closed under
extensions and submodules. 
\end{lemma}
\begin{proof}
  It is clear that $\relproj{E^{\perp}}$ is closed under extensions and
  submodules and contains $\fgproj{R}$ and $\cat{E}$. Let us for the
  moment name the smallest category with these properties $\cat{B} $.  
  Let $M\in \relproj{E^{\perp}}$.  The kernel $K$ of the homomorphism from
  $M$ to $M\tensor R_{\cat{E}}$ lies in $\cat{E}$ so it is enough to
  show that $M/K\in \cat{B}$. However, $M/K$ is an $R$  submodule of a
  finitely generated projective module over $R_{\cat{E}}$ and hence of
  some finitely generated free module over $R_{\cat{E}}$. However, by
  theorem \ref{th:structfpind}, $R_{\cat{E}}^{n}$ is a direct limit of $R$
  modules $N_{i}$ for which we have a short exact sequence
  $\tses{R^{n}}{N_{i}}{E_{i}}$ where $E_{i}\in \cat{E}$. So
  $M/K\subset N_{i}$ for some $i$ and $N_{i}\in \cat{B}$ which implies
  $M/K \in \cat{B}$.  
\end{proof}

Next we point out the connection between $\fac{E}$ and
$\relproj{E^{\perp}}$. 

\begin{lemma}
\label{lem:Tindisproj}
Let $T\in \fac{E}$. Then $\Tor_{1}^{R}(T, R_{\cat{E}}) $ is a finitely
generated projective module over $R_{\cat{E}}$. 

In fact, if we have a short exact sequence $\tses{K}{E}{T}$ where
$E\in \cat{E}$ then $\Tor_{1}^{R}(T, R_{\cat{E}})\cong K \tensor R_{\cat{E}}$. 
\end{lemma}
 \begin{proof}
We begin with the second statement which follows at once because
$\Tor_{1}^{R}(E, R_{\cat{E}}) =0= E \tensor R_{\cat{E}}$.

Now the first statement follows because we can find $E\in \cat{E}$ and
a surjection from $E$ onto $T$. 
\end{proof}

We come to our first main theorem.

\begin{theorem}
\label{th:charfgproj}
Let $R$ be a hereditary ring and let $\cat{E}$ be a well-placed
subcategory of $\fpmod{R}$. Then every finitely generated projective
module over $ R_{\cat{E}}$ is isomorphic to a module of the form
$M\tensor R_{\cat{E}}$ for some $M\in \relproj{E^{\perp}}$. 

Therefore, $P_{\directsum}(R_{\cat{E}})\cong
\abmon{\relproj{E^{\perp}}}/\{[E]=0, E\in \cat{E}\}$.  
\end{theorem}
\begin{proof}
Every finitely presented module over $R_{\cat{E}}$ is induced from
some finitely presented module over $R$ which applies to any finitely
generated projective module over $R_{\cat{E}}$. Moreover, we can
assume that it is of the form $M\tensor R_{\cat{E}}$  where $M$ embeds
in $M\tensor R_{\cat{E}}$. But as we saw in the argument of lemma
\ref{lem:PEptoproj}, this implies that $M\in \relproj{E^{\perp}}$.  

Now suppose that $M\tensor R_{\cat{E}} \cong N \tensor
R_{\cat{E}}$. Let $K_{1}$  and $K_{2}$ be the kernels of the
homomorphisms from $M$ to $M\tensor R_{\cat{E}}$ and from $N$ to $N
\tensor R_{\cat{E}}$ respectively. Then $K_{1},K_{2} \in \cat{E}$ and 
$M\tensor R_{\cat{E}}\cong M/K_{1}\tensor R_{\cat{E}}$ whilst
$N\tensor R_{\cat{E}}\cong N/K_{2}\tensor R_{\cat{E}}$. So
$[M]=[M/K_{1}] + [K_{1}]$ and $[N]=[N/K_{2}]+ [K_{2}]$.

We saw in theorem \ref{th:chariso} that $M/K_{1}\tensor R_{\cat{E}} \cong
N/K_{2} \tensor R_{\cat{E}}$ if and only if there exist two short
exact sequences $\tses{M/K_{1}}{L}{E}$ and $\tses{N/K_{2}}{L}{E'}$
where $E, E'\in \cat{E}$. So $[M/K_{1}]+ [E] = [N/K_{2}] + [E']$. 

It follows that $[M] + [E] + [K_{2}]=  [N] + [E'] + [K_{1}]$ in
$\abmon{\relproj{E^{\perp}}}$ which shows that the relations are as stated.  
\end{proof}

We now link in what happens to the category $\fac{E}$.

\begin{lemma}
\label{lem:remtopE}
Let $M \in \fac{E}$. Then $M$ has a unique submodule $T$ such that
$M/T \in \cat{E}$ and $\Hom(T,E) =0$ for every $E\in \cat{E}$. $T$ is
itself in $\fac{E}$. Further, $\Tor_{1}^{R}(T, R_{\cat{E}}) \cong
\Tor_{1}^{R}(M, R_{\cat{E}}) $.
\end{lemma}
\begin{proof}
Suppose that $\Hom(M,E) \neq 0$. Then choose some $F \in \cat{E}$ that
maps onto $M$ and some nonzero $\map{f}{M}{E}$ where $E\in
\cat{E}$. The image of $f$ in $E$ is also the image of the composition
of the surjection from $F$ to $M$ with $f$ and since $\cat{E}$ is
closed under images, the image of $f$ lies in $\cat{E}$ and we obtain
a short exact sequence $\tses{M_{1}}{M}{I_{1}}$ where $0\neq I_{1} \in
\cat{E}$. Applying $\tensor R_{\cat{E}} $ to this
short exact sequence, we see that $M_{1} \tensor R_{\cat{E}} =0$ and
so $M_{1} \in \fac{E}$ and also $\Tor_{1}^{R}(M_{1}, R_{\cat{E}}) \cong
\Tor_{1}^{R}(M, R_{\cat{E}}) $.  

Iterating this procedure either gives a descending chain of submodules of
$M$, $M\supset M_{1} \supset \dots \supset M_{n} \supset$ or else some
$\Hom(M_{i},\ ) $ vanishes on $\cat{E}$. Since $M$  satisfies the
descending chain condition on bound submodules we cannot have an
infinite descending chain and therefore we obtain some $M_{i}$ where
$\Hom(M_{i},\ ) $ vanishes on $\cat{E}$. By construction, $M/M_{i} \in
\cat{E}$. So we  take $T= M_{i}$. However, the uniqueness
of $T$ is clear when it exists since $T$ must lie in the kernel of any
homomorphism from $M$ to a module in $\cat{E}$. 

The isomorphism of $\Tor_{1}^{R}(T, R_{\cat{E}}) $  with
$\Tor_{1}^{R}(M, R_{\cat{E}})$ follows from the composition of the
isomorphisms of $\Tor_{1}^{R}(M_{j}, R_{\cat{E}})$ with
$\Tor_{1}^{R}(M_{j+1}, R_{\cat{E}})$ for each $j$. 
\end{proof}

The point of this lemma is to remove surplus copies of modules in
$\cat{E}$ from the top of modules in $\fac{E}$ as a dual process to
the removal of modules in $\cat{E}$ that occur as submodules in
modules in $\relproj{E^{\perp}}$.

\begin{lemma}
\label{lem:chartoriso}
Let $M,N$ be modules in $ \fac{E}$ such that $\Hom(M,\ ) $ and
$\Hom(N,\ )$ both vanish on $\cat{E}$. Then $\Tor_{1}^{R}(M,
R_{\cat{E}}) \cong \Tor_{1}^{R}(N, R_{\cat{E}})$ if and only if there
exist two short exact sequences $\tses{K}{E_{1}}{M}$ and
$\tses{K}{E_{2}}{N}$ where $E_{1}, E_{2}\in \cat{E}$ if and only if there
exist two short exact sequences $\tses{E}{L}{M}$ and $\tses{E'}{L}{N}$
where $ E,E'\in \cat{E}$.
\end{lemma}
\begin{proof}
  Assume that $\Tor_{1}^{R}(M, R_{\cat{E}}) \cong \Tor_{1}^{R}(N,
  R_{\cat{E}})$. We choose short exact sequences $\tses{A}{F}{M}$ and
  $\tses{B}{F'}{N} $ where $F,F' \in \cat{E}$ where $A,B$ are
  torsion-free with respect to the torsion theory generated by
  $\cat{E}$. We know that $A \tensor R_{\cat{E}} \cong \Tor_{1}^{R}(M,
  R_{\cat{E}}) \cong \Tor_{1}^{R}(N, R_{\cat{E}}) \cong B \tensor
  R_{\cat{E}}$ and so by \ref{th:chariso}, there exist two short exact
  sequences $\tses{A}{K}{G_{1}}$ and $\tses{B}{K}{G_{2}}$ where $G_{i}
  \in \cat{E}$. We form the pushout of the short exact sequence
  $\tses{A}{F}{M}$  along the map from $A$ to $K$ and the pushout of
  the short exact sequence $\tses{B}{F'}{N}$ along the map from $B$ to
  $K$ to obtain two short exact sequences $\tses{K}{E_{1}}{M}$ and
  $\tses{K}{E_{2}}{N}$ where $E_{i} \in \cat{E}$ because $E_{1}$ is an
  extension of $G_{1}$ on $F$ and $E_{2}$ is an extension of $G_{2}$
  on $F'$. Thus the first condition implies the second.

  Now assume that there exist two short exact sequences
  $\tses{K}{E_{1}}{M}$ and $\tses{K}{E_{2}}{N}$ where $E_{1}, E_{2}\in
  \cat{E}$. Forming the common pushout of these two short exact
  sequences gives a module $L$ and two short exact sequences
  $\tses{E_{2}}{L}{M}$ and $\tses{E_{1}}{L}{N}$ which completes the
  proof that the second condition implies the third.

  Finally, if we assume that there exist two short exact sequences
  $\tses{E}{L}{M}$ and $\tses{E'}{L}{N}$ where $ E,E'\in \cat{E}$ then
  applying $\tensor R_{\cat{E}}$ shows that $\Tor_{1}^{R}(M,
  R_{\cat{E}}) \cong \Tor_{1}^{R}(L, R_{\cat{E}}) \cong
  \Tor_{1}^{R}(N, R_{\cat{E}})$ and demonstrates that the third
  condition implies the first.
\end{proof}

We are in a position to prove the results we want about the category
$\fac{E}$.  

\begin{theorem}
\label{th:relsonA+fac}
Let $R$ be a hereditary ring and let $\cat{E}$ be a well-placed
subcategory of $\fpmod{R}$. Then $\Tor_{1}^{R}(\ , R_{\cat{E}}) $
gives an exact functor from $\fac{E}$ to $\fgproj{R_{\cat{E}}}$. Its
image is the same as the image of the full subcategory of finitely
presented submodules of modules in $\cat{E}$. 

The relations imposed on $\abmon{\fac{E}}$ by the induced map to
$P_{\directsum}(R_{\cat{E}})$ are precisely $\{[E]=0; E\in \cat{E}\}$.
\end{theorem}
\begin{proof}
We have already noted that $\Tor_{1}^{R}(\ , R_{\cat{E}})$ gives an
exact functor from $\fac{E}$ to $\fgproj{R_{\cat{E}}}$. Choosing a
surjection from some module $E\in \cat{E}$  onto a given $M \in
\fac{E}$ with kernel $K$ identifies $\Tor_{1}^{R}(M, R_{\cat{E}})$
with $K \tensor R_{\cat{E}}$ and proves that the image of this functor
is the image of the full subcategory of finitely presented submodules
of modules in $\cat{E}$. 

Finally, we see that lemmas \ref{lem:chartoriso} and \ref{lem:remtopE}
give the relations imposed on $\abmon{\fac{E}}$ to be as required.
\end{proof}
 
\begin{theorem}
\label{th:fgpis+nice}
Let $R$ be a hereditary ring and let $\cat{E}$ be a well-placed
subcategory of $\fpmod{R}$  all of whose modules are bound. Let $S$ be
the set of modules simple as objects of $\cat{E}$. Then every finitely
generated projective module over $R_{\cat{E}}$ is isomorphic to a
direct sum of modules $M \tensor R_{\cat{E}}$  where either $M$ is a
finitely generated projective module over $R$ or else $M$ is a
submodule of a module in $S$. 
\end{theorem}
\begin{proof}
  We may take our finitely generated projective module over
  $R_{\cat{E}}$ to be of the form $K \tensor R_{\cat{E}}$ where $K$ is
  a finitely presented module over $R$ torsion-free with respect to
  the torsion theory generated by $\cat{E}$ so that $K$ is a submodule
  of $K \tensor R_{\cat{E}}$ which is a submodule of $R_{\cat{E}}^{n}$
  for some $n$. Since $R_{\cat{E}}^{n}$ is a directed union of modules
  $F_{t}$ where $F_{t}\supset R^{n}$ and $F_{t}/R^{n} \in \cat{E}$, we
  know that $K$ is a submodule of some $F_{t}$. But $F_{t}$ has a
  filtration whose subfactors are either finitely generated free or
  else isomorphic to a module in $S$. Intersecting this filtration
  with $K$ gives a filtration of $K$ where the subfactors are
  isomorphic to submodules of free modules (so projective) or else
  submodules of modules in $S$. Moreover, these subfactors are
  finitely presented because $R$ is hereditary and so coherent. 

  Because both of these kinds of modules induce to finitely generated
  projective module over $R_{\cat{E}}$, it follows that this
  filtration is split by $\tensor R_{\cat{E}}$ and so $K \tensor
  R_{\cat{E}}$ is isomorphic to the direct sum of the modules induced
  from the subfactors which gives what we set out to prove.
\end{proof}

In order to get further information, we introduce a new kind of module
in $\relproj{E^{\perp}}$. We say that a module in some full subcategory of
$\fpmod{R}$ is late in that category if and only if any map in the
full subcategory from it is injective. If the category is an exact
abelian subcategory then this is equivalent to being a simple object
in the category but in general it is not. Such a module may have
subobjects in the subcategory if the subcategory is not closed under
factors. We have no guarantee that late objects exist; however, when
they do they can be useful.

\begin{lemma}
\label{lem:indindislate}
Let $M$ be torsion-free with respect to the torsion theory generated
by $\cat{E}$ and assume that $M$  has no factor in $\cat{E}$. Then if
$M \tensor R_{\cat{E}}$ is an indecomposable finitely generated
projective module over $R_{\cat{E}}$, $M$ is a late object in
$\relproj{E^{\perp}}$.    
\end{lemma}
 \begin{proof}
Since $M$ is torsion-free, $M$ embeds in $M \tensor R_{\cat{E}}$ and
so $M\in \relproj{E^{\perp}}$. Suppose that we have a map in
$\relproj{E^{\perp}}$, $\map{f}{M}{N}$. Since $\relproj{E^{\perp}}$ is closed
under submodules, we have a short exact sequence
$\tses{\ker(f)}{M}{\im(f)}$ all of whose terms lie in
$\relproj{E^{\perp}}$. Applying $\tensor R_{\cat{E}}$ shows that $M \tensor
R_{\cat{E}}\cong \im(f) \tensor R_{\cat{E}} \directsum \ker(f) \tensor
R_{\cat{E}}$ and since $M \tensor R_{\cat{E}}$ is indecomposable
either $\im(f) \tensor R_{\cat{E}}=0$ or else $\ker(f) \tensor
R_{\cat{E}}=0$. In the first case, this implies that $\im(f) \in
\cat{E}$  and so $\im(f)=0$ since $M$ has no factor in $\cat{E}$ and
this means that $f$ is the zero map. In
the second case, this implies that $\ker(f) \in \cat{E}$ and so
$\ker(f)=0$ which means that $f$ is injective as stated. 
\end{proof}

Thus if we happen to know that every finitely generated projective
module over $R_{\cat{E}}$ is a direct sum of indecomposable finitely
generated projective modules we will have a large supply of late
objects in $\relproj{E^{\perp}}$. 

We need the dual concept for the category $\fac{E}$. We say that a
module in a full subcategory is an early object in the subcategory if
and only if any nonzero map in the subcategory to it must be
surjective as a homomorphism of modules. We simply state the next
lemma since its proof is no different from the preceding lemma.

\begin{lemma}
\label{lem:ind2early}
Let $M \in \fac{E}$ have no factors nor submodules in $\cat{E}$. Then
if $\Tor_{1}^{R}(M, R_{\cat{E}})$ is an indecomposable finitely
generated projective module over $R_{\cat{E}}$, $M$ is an early object
in $\fac{E}$.   
\end{lemma}

Let $S$ be the set of simple objects in $\cat{E}$. We say that the
modules $A$ and $B$ are directly $S$-related if and only if there
exists either a short exact sequence $\tses{A}{S_{i}}{B}$ or a short
exact sequence $\tses{B}{S_{i}}{A}$ where $S_{i} \in S$. We say that
$A$ and $B$ are $S$-related if they are equivalent under the
equivalence relation generated by this relation. Clearly, if $A,B \in
\relproj{E^{\perp}}$ and are $S$-related then $A \tensor R_{\cat{E}}
\cong B \tensor R_{\cat{E}}$. We should like a converse to this
observation since attempting to find extensions of $A$ and $B$ by
modules in $\cat{E}$ that are isomorphic is hard whereas we may well
have a clear grasp of the submodule structure of the modules in $S$
and thus an ability to calculate which modules are $S$-related. The
next theorem provides us with a partial converse which in practice is
frequently decisive.

\begin{theorem}
\label{th:isorel}
Let $R$ be a hereditary ring and let $\cat{E}$ be a full subcategory
of $\fpmod{R} $ all of whose modules are bound. Let $S$ be the set of
simple objects in $\cat{E}$.  Let $A,B$ be late objects in
$\relproj{E^{\perp}}$ such that $A \tensor R_{\cat{E}} \cong B \tensor
R_{\cat{E}}$. Then either $A$ is $S$-related to $B$ or else one of $A$
and $B$ is
$S$-related to a module in $\relproj{E^{\perp}}$ that is not late or to a
module in $\fac{E}$ that is not early, either of which imply that $A
\tensor R_{\cat{E}}$ is a decomposable projective module over
$R_{\cat{E}}$. 

Thus if in addition $A \tensor R_{\cat{E}}$ is an indecomposable
finitely generated projective module over $R_{\cat{E}}$ then $A$ and
$B$ are $S$-related. 
\end{theorem}
 \begin{proof}
   Since $A \tensor R_{\cat{E}} \cong B \tensor R_{\cat{E}}$, there
   exist short exact sequences $\tses{A}{C}{E}$ and $\tses{B}{C}{F}$
   where $E,F\in \cat{E}$. Since $\cat{E}$ is a finite length category
   we proceed by induction on the sum of the lengths of $E$ and $F$ as
   objects in $\cat{E}$. Neither of these lengths can be zero since
   this forces one of $A$ and $B$ to be a submodule of the other and
   hence makes them isomorphic. 

   Let $\tses{F'}{F}{T}$ be a short exact sequence where $T$ is a
   simple object in $\cat{E}$. Consider the composition of the maps
   from $A$ to $C$ to $F$ to $T$. If this is $0$, we take $C'$ to be
   the kernel of the surjection from $C$ to $T$ and $E'$ to be the
   kernel of the induced surjection from $E$ to $T$ and obtain the
   short exact sequences $\tses{A}{C'}{E'}$ and $\tses{B}{C'}{F'}$
   where $E'$ and $F'$ have shorter length as objects in
   $\cat{E}$. Otherwise, the map from $A$ to $T$ is nonzero. It cannot
   be surjective since $A$ is late and not in $\cat{E}$. So it must be
   injective because $A$ is late. Let $A'$ be the factor so that $A$
   is directly $S$-related to $A'$; we note that $A' = C/A+C'$, $A\cap
   C'=0$ and so we have a short exact sequence $\tses{F'}{C/A+B}{A'}$
   where $F' \in \cat{E}$ of length $1$ less than the length of
   $F$. Interchanging the role of $A$ and $B$, we obtain $B'$ directly
   $S$-related to $B$, and a short exact sequence
   $\tses{E'}{C/A+B}{B'}$ where $E' \in \cat{E}$ and has length $1$
   less than the length of $E$. 

   We now note that we have obtained a dual situation to the one we
   began with for $A$ and $B$ whenever both $A'$ and $B'$ are early
   objects in $\fac{E}$ but since the sum of the lengths of $E'$
   and $F'$ is $2$ less than the sum of the lengths of $E$ and $F$,
   induction and the dual argument to the one above apply to complete
   the proof.
\end{proof}

\end{document}